\documentclass[11pt,a4paper, twoside]{amsart}
\usepackage[latin1]{inputenc}

\usepackage{tikz}

\usepackage{amsmath}
\usepackage{amsfonts}
\usepackage{amssymb}
\usepackage{latexsym}

\usepackage{hyperref}
\hypersetup{backref, pdfpagemode=FullScreen, colorlinks=true,
  citecolor=magenta, linkcolor=cyan, urlcolor=blue}

\usepackage{mathtools}
\mathtoolsset{showonlyrefs}

\usepackage{graphicx}
\usepackage{xcolor}
\usepackage{graphicx} 

\usepackage{version} 

\theoremstyle{change}
\newtheorem{theo}{Theorem}[section]
\newtheorem*{theo*}{Theorem}

\newtheorem{prop}[theo]{Proposition}

\newtheorem{rem}[theo]{Remark}

\numberwithin{equation}{section}

\newcommand{\inte}{\mathrm{int}\,}  
\newcommand{\conv}{\mathrm{conv}\,} 

\newcommand{\vol}{\mathrm{vol}\,} 

\newcommand{\Kn}{{\mathcal K}^n} 

\newcommand{\Ksn}{{\mathcal K}_{(s)}^n} 
\newcommand{\Knull}{{\mathcal K}_{(o)}^n} 

\newcommand{\R}{\mathbb{R}} 

\newcommand{\Z}{\mathbb{Z}} 

\newcommand{\disuni}{\mathbin{\setbox0\hbox{$\bigcup$}\rlap{\copy0}\raise.3%
  \ht0\hbox to \wd0{\hfil$\cdot$\hfil}}}

\newcommand{\vu}{{\boldsymbol u}}
\newcommand{\vv}{{\boldsymbol v}}

\newcommand{\vx}{{\boldsymbol x}}

\newcommand{\vnull}{{\boldsymbol 0}}

\def\Kstw{\mathcal{K}_{(s)}^2}

\def\Solution{\frac{1}{3-\sqrt{3}}}

\def\Solutiontwo{\frac{\cos\frac{\pi}{5}-\sin\frac{\pi}{5}+\sin\frac{2\pi}{5}-\cos\frac{2\pi}{5}}{2\sin\frac{\pi}{5}}}

 \newcommand{\bnorm}[1]{\left\|#1\right\|}

\DeclareMathOperator{\St}{St}

\begin{document}
	
\title{On the lattice point covering problem in dimension 2}
\author{Fei Xue}
\address{Institut f\"ur Mathematik, Technische Universit\"at Berlin,
	Sekr. Ma 4-1, Strasse des 17 Juni 136, D-10623 Berlin, Germany}
\email{xue@math.tu-berlin.de}
\thanks{The research of the author was supported by a PhD scholarship of the Berlin
	Mathematical School}

\begin{abstract}
	In this paper we study the lattice point covering property of some regular polygons in dimension 2.
\end{abstract}

\maketitle

\section{Introduction}

Let $\Kn$ be the set of all convex bodies, i.e., compact convex sets, in the $n$-dimensional Euclidean space $\R^n$ with non-empty interior. 
We denote by $\Knull\subset \Kn$ the set of all  convex bodies, having the origin as an interior point, i.e., $\vnull\in\inte(K)$, and by $\Ksn\subset\Knull$ those bodies which are symmetric with respect to $\vnull$, i.e., $K=-K$.

We say that a convex body $K\in\Kn$ has the lattice point covering property, i.e., contains a lattice point of $\Z^n$ in any position, if each translation and rotation of $K$ contains at least a lattice point of $\Z^n$.

There are several beautiful results about the lattice point covering property.

\begin{theo}[Henk\&Tsintsifas,\cite{HeTs07}]
	Let $\mathcal{E}\subset\R^n$ be an ellipsoid with semi-axes $\alpha_i$, $1\leq i\leq n$. The following statements are equivalent:
	\begin{equation}
	\begin{split}
	(i)& \text{ $\mathcal{E}$ contains a lattice point of $\Z^n$ in any position,}\\
	(ii)& \text{ $\sum_{i=1}^{n}\frac{1}{\alpha_i^2}\leq 4$,}\\
	(iii)& \text{ $\mathcal{E}$ contains a cube of edge length $1$.}
	\end{split}
	\end{equation}
\end{theo}

\begin{theo}[Niven\&Zuckerman,\cite{NiZu67}]
	A triangle with sides of lengths $a,b,c$, with $a\leq b\leq c$, has the lattice point covering property if and only if $2\Delta(c-1)\geq c^2$ where $\Delta$ is the area of the triangle.
\end{theo}

\begin{theo}[Niven\&Zuckerman,\cite{NiZu67}]
	Let $a$ and $b$ be the distances between the pairs of opposite sides, say with $a\leq b$, of a parallelogram $ABCD$ with an interior angle $\gamma\leq\pi/2$. The parallelogram has the lattice point covering property if and only if $a\geq 1$ and one of the following conditions holds:\\
	(i) $b\geq\sqrt{2}$;\\
	(ii) $b\leq\sqrt{2}$ and $\alpha+\beta+\gamma\leq\pi/2$, where $\alpha=\arccos(a/\sqrt{2})$ and $\beta=\arccos(b/\sqrt{2})$.
\end{theo}

Let $K\in\Kn$. Denote by $Z(K)$ the lattice point covering radius of $K$, i.e., the smallest positive number $r$, such that $rK$ has the lattice point covering property.

We here concern the lattice point covering properties of regular polygons. Denote by
\begin{equation*}
H_n=\conv\left\{\left(\cos\left(\frac{2k\pi}{n}\right),\sin
\left(\frac{2k\pi}{n}\right)\right):k=0,1,\cdots,n-1\right\}
\end{equation*}
the regular $n$-gon. 

Our main result is:

\begin{theo}
	Let $t_i>0$, $i\in\Z_+$.
	
	(1) The following statements are equivalent:
	\begin{equation}
	\begin{split}
	i) &\, \text{$t_{4n}\cdot H_{4n}$ contains a lattice point of $\Z^2$ in any position,}\\
	ii) &\, \text{$t_{4n}\cdot H_{4n}$ contains a ball with radius $\frac{\sqrt{2}}{2}$,}\\
	iii) &\, t_{4n}\geq Z(H_{4n})=\frac{\frac{\sqrt{2}}{2}}{\cos\frac{\pi}{2n}}.
	\end{split}
	\end{equation}
	
	(2) The following statements are equivalent for $n=1,2$:
	\begin{equation}
	\begin{split}
	i) &\,\text{$t_{4n+2}\cdot H_{4n+2}$ contains a lattice point of $\Z^2$ in any position,}\\
	ii) &\, \text{$t_{4n+2}\cdot H_{4n+2}$ contains $[-\frac{1}{2},\frac{1}{2}]^2$,}\\
	iii) &\, t_6\geq Z(H_6)=\Solution\approx 0.788675\dotsc, \\
	&\, t_{10}\geq Z(H_{10})=\Solutiontwo\approx 0.734342\dotsc.
	\end{split}
	\end{equation}
\end{theo}


This paper is organized as follows: in the Section 2 we will introduce the covering radius and gauge function; then we will give a necessary and a sufficient condition of the lattice point covering property in the Section 3 and the Section 4; the content of the Section 5 will be the proof of the main theorem.

\section{covering radius}

The covering radius of $K\in\Knull$ with respect to $\Z^n$ is denoted by
\begin{equation*}
c(K)=c(K,\Z^n)=\min\{\lambda>0:\lambda K+\Z^n=\R^n\}.
\end{equation*}

The gauge function $\bnorm{\cdot}$ associated to a $K\in\Knull$ is the function 
\begin{equation*}
\bnorm{\cdot}_{K}:\R^n\rightarrow[0,\infty)
\end{equation*}
defined by
\begin{equation*}
\bnorm{\vv}_K=\min\{t>0:\vv\in tK\}.
\end{equation*}

\begin{theo}
	Let $K\in\Knull$. Then $K$ contains a lattice point of $\Z^n$ in any position if and only if, for any $o(K)$ rotation of $K$, $c(o(K))\leq 1$.
	\label{theo:rotation}
\end{theo}

\begin{proof}
	If $c(o(K))>1$ for some rotation $o(K)$, then there exists a point $\vx\in\R^n$, such that for every $\vu\in\Z^2$,
	\begin{equation*}
	\bnorm{\vx-\vu}_{o(K)}>1\Longleftrightarrow \bnorm{\vu-\vx}_{-o(K)}>1.
	\end{equation*}
	Therefore, $-o(K)+\vx$ does not contain a lattice point of $\Z^2$.
	
	If $c(o(K))\leq 1$ for any rotation $o(K)$, then for any point $\vx\in\R^n$, since $o(K)+\Z^n=\R^n$, there exists a lattice point $\vu\in\Z^2$, such that
	\begin{equation*}
	\bnorm{\vx-\vu}_{o(K)}\leq 1\Longleftrightarrow \bnorm{\vu-\vx}_{-o(K)}\leq 1.
	\end{equation*}
	So, $-o(K)+\vx$ contain a lattice point $u$.
\end{proof}

Therefore, the lattice point covering property of a convex body depends on the covering radius of all rotations of this convex body.

\section{necessary condition}

According to the knowledge of the lattice covering for a centrally symmetric convex body, we have:

\begin{theo}[I.F\'ary,\cite{Fa50}]
	Let $K\in\Kstw$, such that $K+\Z^2$ is a lattice covering. Then $K$ contains a spacefiller $L$, i.e., a parallelogram or a centrally symmetric hexagon, such that $L+\Z^2$ is a lattice tiling.
	\label{theo:covering}
\end{theo}

Since the lattice point covering property depends on the lattice covering of all rotations, we have:

\begin{theo}
	Let $K\in\Kstw$. Then $K$ contains a lattice point of $\Z^2$ in any position, if and only if $o(K)$ contains a spacefiller, i.e., parallelogram or a centrally symmetric hexagon, with respect to the lattice $\Z^2$ for any rotation $o(K)$.
\end{theo}

\begin{proof}
	Cf. Theorem \ref{theo:rotation} and \ref{theo:covering}.
\end{proof}

We here call it a necessary condition of the lattice point covering property, because this condition has to hold, but is uneasy to check.

\section{sufficient condition}

For a planar convex body, it is possible to check some inscribed parallelograms, i.e., by checking the Steiner symmetrization of the convex body. We will give a sufficient condition of the lattice point covering property in this way.


The Steiner symmetrization of $K\in\mathcal{K}^2$ with respect to $\{\vx\in\R^2:x_2=0\}$, denoted by $\St_1(K)$, is a convex body symmetric with respect to $\{\vx\in\R^2:x_2=0\}$, such that for each line $l$ vertical to $\{\vx\in\R^2:x_2=0\}$,
\begin{equation}
\vol_1(K\cap l)=\vol_1(\St_1(K)\cap l).
\end{equation}
For more information on the Steiner symmetrization, we refer to \cite[Section 9.1]{Gr07}. It is obvious that $\St_1(K)\subset\St_1(L)$ for two convex bodies $K\subset L$.

\begin{theo}
	Let $K\in\Kstw$. If for each rotation $o(K)$ of $K$, $\St_1(o(K))$ contains $[-\frac{1}{2},\frac{1}{2}]^2$, then $K$ contains a lattice point of $\Z^2$ in any position.
	\label{theo:cube}
\end{theo}

\begin{proof}
	Notice that if $\St_1(o(K))$ contains $[-\frac{1}{2},\frac{1}{2}]^2$, then $o(K)$ contains a parallelogram in the form of $L=\conv\{(a,\frac{1}{2}),(a-1,\frac{1}{2}),(-a,-\frac{1}{2}),(1-a,-\frac{1}{2})\}$, which is a spacefiller with respect to $\Z^2$.
	Therefore $K$ has the lattice point covering property (cf. Theorem \ref{theo:rotation}).
\end{proof}
	
Somehow, we also have the following proposition of lattice covering.

\begin{prop}
	Let $K\in\Kstw$. If $K$ is symmetric with respect to $\{\vx\in\R^2:x_1=0\}$ and $\{\vx\in\R^2:x_2=0\}$, then $K+\Z^2=\R^2$ if and only if $K$ contains $[-\frac{1}{2},\frac{1}{2}]^2$.
	\label{prop:iff}
\end{prop}

\begin{proof}
	If $K$ contains $[-\frac{1}{2},\frac{1}{2}]^2$, then $K+\Z^2$ is a lattice covering. Otherwise, if $K$ does not contain $[-\frac{1}{2},\frac{1}{2}]^2$, i.e., $(\frac{1}{2},\frac{1}{2})\notin K$, then since $K$ is symmetric with respect to $\{\vx\in\R^2:x_1=0\}$ and $\{\vx\in\R^2:x_2=0\}$, $K$ does not contain any point of $(\frac{1}{2},\frac{1}{2})+\Z^2$, thus $K+\Z^2$ does not contain $(\frac{1}{2},\frac{1}{2})+\Z^2$.
\end{proof}

\section{proof of the main theorem}

In this section we discuss the lattice point covering property of some regular polygons. The proofs are based on Theorem \ref{theo:cube} and have the following steps:

1. Prove that the Steiner symmetrizations of all rotations of the convex body contain $[-\frac{1}{2},\frac{1}{2}]^2$.

2. Prove that a smaller copy of the convex body does not have the lattice point covering property.

We first look at the regular $4n$-gon.

\begin{theo}
	Let $t>0$. The following statements are equivalent:
	\begin{equation}
	\begin{split}
	i) & \text{ $t\cdot H_{4n}$ contains a lattice point of $\Z^2$ in any position,}\\
	ii) & \text{ $t\cdot H_{4n}$ contains a ball with radius $\frac{\sqrt{2}}{2}$,}\\
	iii) &\, t\geq Z(H_{4n})=\frac{\frac{\sqrt{2}}{2}}{\cos\frac{\pi}{2n}}.
	\end{split}
	\end{equation}
\end{theo}


\begin{proof}
	If $t\cdot H_{4n}$ contains $B_2(\frac{\sqrt{2}}{2})$, then each rotation $o(t\cdot H_{4n})$ also contains $B_2(\frac{\sqrt{2}}{2})$. Notice that $B_2(\frac{\sqrt{2}}{2})$ contains $[-\frac{1}{2},\frac{1}{2}]^2$, therefore $o(t\cdot H_{4n})+\Z^2$ is always a lattice covering, thus $t\cdot H_{4n}$ has the lattice point covering property (cf. Theorem \ref{theo:rotation}).
	
	If $t\cdot H_{4n}$ does not contain $B_2(\frac{\sqrt{2}}{2})$, then $o(t\cdot H'_{4n},\frac{\pi}{4})+(\frac{1}{2},\frac{1}{2})$, does not contain any lattice point of $\Z^2$, where $o(t\cdot H'_{4n},\frac{\pi}{4})$ is the rotation of $t\cdot H'_{4n}$ by angle $\frac{\pi}{4}$.
\end{proof}

Then we look at the regular hexagon.

\begin{theo}
	Let $t>0$. The following statements are equivalent:
	\begin{equation}
	\begin{split}
	i) &\,\text{$t\cdot H_{6}$ contains a lattice point of $\Z^2$ in any position,}\\
	ii) &\, \text{$t\cdot H_{6}$ contains $[-\frac{1}{2},\frac{1}{2}]^2$,}\\
	iii) &\, t\geq Z(H_6)=\Solution.
	\end{split}
	\end{equation}
	\label{theo:hexagon}
\end{theo}


\begin{proof}
	Denote by $o(K,\theta)$ the counterclockwise rotation of $K$ by angle $\theta$, i.e.,
	\begin{equation*}
	o(H_n,\theta)=\conv\{(\cos(\frac{2k\pi}{n}+\theta),\sin(\frac{2k\pi}{n}+\theta)):k=0,1,\cdots,n-1\}.
	\end{equation*}
	For the symmetric reason, the case $\frac{\pi}{12}\leq\theta\leq\frac{\pi}{6}$ is actually symmetric to the case $0\leq \theta\leq \frac{\pi}{12}$ with respect to the line $\{\vx\in\R^2:x_2=x_1\}$. 
	We are going to prove that $\St_1(o(\Solution H_{6},\theta))$ contains $[-\frac{1}{2},\frac{1}{2}]^2$ for $0\leq \theta\leq \frac{\pi}{12}$ (cf. Theorem \ref{theo:covering}). 
	
	By calculation,
	\begin{equation*}
	\St_1(o(H_6,\theta))=\{(\pm\cos\theta,0),(\pm\cos(\theta-\frac{\pi}{3}),\pm\frac{\sin\frac{\pi}{3}}{2\sin(\theta+\frac{\pi}{6})}),(\pm\cos(\theta+\frac{\pi}{3}),\pm\frac{\sqrt{3}}{2\cos\theta})\},
	\end{equation*}
	which by the way is also symmetric with respect to $\{\vx\in\R^2:x_1=0\}$. In order to check whether $\St_1(o(\Solution H_{6},\theta))$ contains $[-\frac{1}{2},\frac{1}{2}]^2$ for $0\leq \theta\leq \frac{\pi}{12}$, notice that when $0\leq\theta\leq\frac{\pi}{12}$, it holds 
	\begin{equation}
	\cos(\theta+\frac{\pi}{3})\leq\frac{\sqrt{3}}{2\cos\theta}
	\label{hexagon1}
	\end{equation}
	(cf. Proposition \ref{prop:A1}), therefore the line $\{\vx\in\R^2:x_2=x_1\}$ may intersects the boundary of $\St_1(o(H_6,\theta))$ with the edge
	\begin{equation*}
	\conv\{(\cos\theta,0),(\cos(\theta-\frac{\pi}{3}),\frac{\sin\frac{\pi}{3}}{2\sin(\theta+\frac{\pi}{6})})\}
	\end{equation*}
	or the edge 
	\begin{equation*}
	\conv\{(\cos(\theta-\frac{\pi}{3}),\frac{\sin\frac{\pi}{3}}{2\sin(\theta+\frac{\pi}{6})}),(\cos(\theta+\frac{\pi}{3}),\frac{\sqrt{3}}{2\cos\theta})\}.
	\end{equation*}
	
	Case 1: $\cos(\theta-\frac{\pi}{3})\leq\frac{\sin\frac{\pi}{3}}{2\sin(\theta+\frac{\pi}{6})}$, i.e., $0\leq\theta\leq\arcsin\left(\frac{\sqrt[4]{3}}{2}\right)-\frac{\pi}{6}$.
	
	In this case, the line $\{\vx\in\R^2:x_2=x_1\}$ intersects the edge
	\begin{equation*}
	\conv\{(\cos\theta,0),(\cos(\theta-\frac{\pi}{3}),\frac{\sin\frac{\pi}{3}}{2\sin(\theta+\frac{\pi}{6})})\}
	\end{equation*}
	with $(s(\theta),s(\theta))$, where
	\begin{equation*}
	\frac{s(\theta)}{s(\theta)-\cos\theta}=\frac{\frac{\sin\frac{\pi}{3}}{2\sin(\theta+\frac{\pi}{6})}}{\cos(\theta-\frac{\pi}{3})-\cos\theta},
	\end{equation*}
	thus
	\begin{equation}
	s(\theta)=\frac{\sin\frac{\pi}{3}\cos\theta}{\sin\frac{\pi}{3}-2\sin^2(\theta+\frac{\pi}{6})+2\cos\theta\sin(\theta+\frac{\pi}{6})}.
	\label{hexagon2}
	\end{equation}
	This function $s(\theta)$ is 
	increasing in $[0,\arcsin\left(\frac{\sqrt[4]{3}}{2}\right)-\frac{\pi}{6}]$ (cf. Proposition \ref{prop:A2}), therefore, 
	\begin{equation*}
	s(\theta)\geq s(0)=\frac{\sqrt{3}}{\sqrt{3}+1},
	\end{equation*}
	and
	\begin{equation*}
	\Solution s(\theta)\geq\frac{1}{2}.
	\end{equation*}
	So $\St_1(o(\Solution H_{6},\theta))$ always contains $[-\frac{1}{2},\frac{1}{2}]^2$ when $\theta\in[0,\arcsin\left(\frac{\sqrt[4]{3}}{2}\right)-\frac{\pi}{6}]$.
	
	Case 2: $\cos(\theta-\frac{\pi}{3})\geq\frac{\sin\frac{\pi}{3}}{2\sin(\theta+\frac{\pi}{6})}$, i.e., $\arcsin\left(\frac{\sqrt[4]{3}}{2}\right)-\frac{\pi}{6}\leq\theta\leq\frac{\pi}{12}$.
	
	In this case, the line $\{\vx\in\R^2:x_2=x_1\}$ intersects the edge
	\begin{equation*}
	\conv\{(\cos(\theta-\frac{\pi}{3}),\frac{\sin\frac{\pi}{3}}{2\sin(\theta+\frac{\pi}{6})}),(\cos(\theta+\frac{\pi}{3}),\frac{\sqrt{3}}{2\cos\theta})\}.
	\end{equation*}
	with $(t(\theta),t(\theta))$, where
	\begin{equation*}
	\frac{2t(\theta)-\frac{\sqrt{3}}{\cos\theta}}{t(\theta)-\cos\left(\theta+\frac{\pi}{3}\right)}=\frac{\frac{\sin\frac{\pi}{3}}{\sin\left(\theta+\frac{\pi}{6}\right)}-\frac{\sqrt{3}}{\cos\theta}}{\cos\left(\theta-\frac{\pi}{3}\right)-\cos\left(\theta+\frac{\pi}{3}\right)},
	\end{equation*}
	i.e.,
	\begin{equation}
	t(\theta)=\frac{2\sqrt{3}\sin\left(\theta+\frac{\pi}{6}\right)+\sqrt{3}\cos\left(\theta+\frac{\pi}{3}\right)}{4\cos\theta\sin\left(\theta+\frac{\pi}{6}\right)+\sqrt{3}}.
	\label{hexagon3}
	\end{equation}
	This function $t(\theta)$ is decreasing in $[\arcsin\left(\frac{\sqrt[4]{3}}{2}\right)-\frac{\pi}{6},\frac{\pi}{12}]$, and in fact decreasing in $[\arcsin\left(\frac{\sqrt[4]{3}}{2}\right)-\frac{\pi}{6},\frac{\pi}{6}]$ (cf. Proposition \ref{prop:A3}), therefore
	\begin{equation*}
	t(\theta)\geq t(\frac{\pi}{12})>t(\frac{\pi}{6})=\frac{\sqrt{3}}{\sqrt{3}+1},
	\end{equation*}
	and
	\begin{equation*}
	\Solution t(\theta)>\frac{1}{2}.
	\end{equation*}
	So $\St_1(o(\Solution H_{6},\theta))$ always contains $[-\frac{1}{2},\frac{1}{2}]^2$ when $\theta\in[\arcsin\left(\frac{\sqrt[4]{3}}{2}\right)-\frac{\pi}{6},\frac{\pi}{12}]$.
	
	To see that $\Solution$ is the minimum number, we refer to Proposition \ref{prop:iff}.
\end{proof}

Now we look at the regular $10$-gon.

\begin{theo}
	Let $t>0$. The following statements are equivalent:
	\begin{equation}
	\begin{split}
	i) &\,\text{$t_{10}\cdot H_{10}$ contains a lattice point of $\Z^2$ in any position,}\\
	ii) &\, \text{$t_{10}\cdot H_{10}$ contains $[-\frac{1}{2},\frac{1}{2}]^2$,}\\
	iii) &\, t_{10}\geq Z(H_{10})=\Solutiontwo.
	\end{split}
	\end{equation}
	\label{theo:tengon}
\end{theo}


\begin{proof}
	Denote by $o(K,\theta)$ the counterclockwise rotation of $K$ by angle $\theta$, i.e.,
	\begin{equation*}
	o(H_n,\theta)=\conv\{(\cos(\frac{2k\pi}{n}+\theta),\sin(\frac{2k\pi}{n}+\theta)):k=0,1,\cdots,n-1\}.
	\end{equation*}
	For the symmetric reason, the case $\frac{\pi}{20}\leq\theta\leq\frac{\pi}{10}$ is actually symmetric to the case $0\leq \theta\leq \frac{\pi}{20}$ with respect to the line $\{\vx\in\R^2:x_2=x_1\}$. 
	We are going to prove that $\St_1(o(\Solutiontwo H_{10},\theta))$ contains $[-\frac{1}{2},\frac{1}{2}]^2$ for $0\leq \theta\leq \frac{\pi}{20}$ (cf. Theorem \ref{theo:covering}).
	
	By calculation,
	\begin{equation*}
	\begin{split}
	\St_1(o(H_{10},\theta))=\{&\pm(\cos\theta,0),\\
	&\left(\pm\cos(\theta-\frac{\pi}{5}),\pm\frac{\sin\frac{\pi}{10}\sin\frac{\pi}{5}}{\sin(\frac{\pi}{10}+\theta)}\right), \left(\pm\cos(\theta+\frac{\pi}{5}),\pm\frac{\sin\frac{\pi}{5}\sin\frac{3\pi}{10}}{\sin(\frac{3\pi}{10}-\theta)}\right),\\
	&\left(\pm\cos(\theta-\frac{2\pi}{5}),\pm\frac{\sin\frac{3\pi}{10}\sin\frac{2\pi}{5}}{\sin(\frac{3\pi}{10}+\theta)}\right), \left(\pm\cos(\theta+\frac{2\pi}{5}),\pm\frac{\sin\frac{2\pi}{5}}{\cos\theta}\right)\}.
	\end{split}
	\end{equation*}
	While $\theta\in[0,\frac{\pi}{20}]$, it holds 
	\begin{equation*}
	\cos(\theta-\frac{\pi}{5})>\frac{\sin\frac{\pi}{10}\sin\frac{\pi}{5}}{\sin(\frac{\pi}{10}+\theta)}\text{ (cf. \ref{prop:A5})},
	\end{equation*}
	\begin{equation*}
	\cos(\theta+\frac{\pi}{5})>\frac{\sin\frac{\pi}{5}\sin\frac{3\pi}{10}}{\sin(\frac{3\pi}{10}-\theta)}\text{ (cf. \ref{prop:A6})},
	\end{equation*}
	\begin{equation*}
	\cos(\theta-\frac{2\pi}{5})<\frac{\sin\frac{3\pi}{10}\sin\frac{2\pi}{5}}{\sin(\frac{3\pi}{10}+\theta)}\text{ (cf. \ref{prop:A7})},
	\end{equation*}
	\begin{equation*}
	\cos(\theta+\frac{2\pi}{5})<\frac{\sin\frac{2\pi}{5}}{\cos\theta}\text{ (cf. \ref{prop:A8})}.
	\end{equation*}
	Therefore, the line $\{\vx\in\R^2:x_2=x_1\}$ intersects $\St_1(o(H_{10},\theta))$ with the edge
	\begin{equation*}
	\conv\left\{\left(\cos(\theta+\frac{\pi}{5}),\frac{\sin\frac{\pi}{5}\sin\frac{3\pi}{10}}{\sin(\frac{3\pi}{10}-\theta)}\right),\left(\cos(\theta-\frac{2\pi}{5}),\frac{\sin\frac{3\pi}{10}\sin\frac{2\pi}{5}}{\sin(\frac{3\pi}{10}+\theta)}\right)\right\}
	\end{equation*}
	at the point $(t(\theta),t(\theta))$, where 
	\begin{equation}
	\frac{t(\theta)-\frac{\sin\frac{3\pi}{10}\sin\frac{2\pi}{5}}{\sin(\frac{3\pi}{10}+\theta)}}{t(\theta)-\cos(\theta-\frac{2\pi}{5})}=\frac{\frac{\sin\frac{\pi}{5}\sin\frac{3\pi}{10}}{\sin(\frac{3\pi}{10}-\theta)}-\frac{\sin\frac{3\pi}{10}\sin\frac{2\pi}{5}}{\sin(\frac{3\pi}{10}+\theta)}}{\cos(\theta+\frac{\pi}{5})-\cos(\theta-\frac{2\pi}{5})}.
	\label{octa}
	\end{equation}
	The function $t(\theta)$ is increasing in $[0,\frac{\pi}{20}]$ (cf. Proposition \ref{prop:A4}), therefore
	\begin{equation*}
	t(\theta)\geq t(0)=\frac{\sin\frac{\pi}{5}}{\cos\frac{\pi}{5}-\sin\frac{\pi}{5}+\sin\frac{2\pi}{5}-\cos\frac{2\pi}{5}},
	\end{equation*}
	and
	\begin{equation*}
	\Solutiontwo t(\theta)\geq\frac{1}{2}.
	\end{equation*}
	So $\St_1(o(\Solutiontwo H_{10},\theta))$ always contains $[-\frac{1}{2},\frac{1}{2}]^2$ when $\theta\in[0,\frac{\pi}{20}].$
	
	To see that $\Solutiontwo$ is the minimum number, we refer to Proposition \ref{prop:iff}.
\end{proof}

\begin{rem}
	$\St_1(o(H_{2n},\theta)),\,\theta\in[0,\frac{\pi}{4n}],$ has the vertices
	$$(\cos\theta,0),$$
	$$\left(\cos(\theta-\frac{2\pi}{n}),\frac{\sin\frac{\pi}{n}\sin\frac{2\pi}{n}}{\sin(\frac{\pi}{n}+\theta)}\right),\left(\cos(\theta+\frac{2\pi}{n}),\frac{\sin\frac{2\pi}{n}\sin\frac{3\pi}{n}}{\sin(\frac{3\pi}{n}-\theta)}\right),$$
	$$\left(\cos(\theta-\frac{4\pi}{n}),\frac{\sin\frac{3\pi}{n}\sin\frac{4\pi}{n}}{\sin(\frac{3\pi}{n}+\theta)}\right),\left(\cos(\theta+\frac{4\pi}{n}),\frac{\sin\frac{4\pi}{n}\sin\frac{5\pi}{n}}{\sin(\frac{5\pi}{n}-\theta)}\right)...$$
	
	We can deal with all regular $(4n+2)$-gons similarly (cf. Theorem \ref{theo:hexagon} and \ref{theo:tengon}).
\end{rem}

\appendix

\section{Some inequalities}

\begin{prop}
	\begin{equation*}
	\cos(\theta+\frac{\pi}{3})\leq\frac{\sqrt{3}}{2\cos\theta}
	\end{equation*}
	for $0\leq\theta\leq\frac{\pi}{12}$.
	\label{prop:A1}
\end{prop}


\begin{proof}
It is equivalent to
\begin{equation*}
\begin{split}
& 2\cos\theta\cos(\theta+\frac{\pi}{3})\leq\sqrt{3}\\
\Longleftrightarrow & \cos\frac{\pi}{3}+\cos(2\theta+\frac{\pi}{3})\leq\sqrt{3}\\
\Longleftrightarrow & \cos(2\theta+\frac{\pi}{3})\leq\sqrt{3}-\frac{1}{2}.
\end{split}
\end{equation*}
\end{proof}

\begin{prop}
	\begin{equation*}
	s(\theta)=\frac{\sin\frac{\pi}{3}\cos\theta}{\sin\frac{\pi}{3}-2\sin^2(\theta+\frac{\pi}{6})+2\cos\theta\sin(\theta+\frac{\pi}{6})}\geq s(0)
	\end{equation*}
	for $0\leq\theta\leq\arcsin\left(\frac{\sqrt[4]{3}}{2}\right)-\frac{\pi}{6}$.
	\label{prop:A2}
\end{prop}


\begin{proof}
Notice that
\begin{equation*}
s(\theta)=\frac{\sqrt{3}\cos\theta}{-3+\sqrt{3}+4\cos^2\theta}.
\end{equation*}
While $\theta\in[0,\arcsin\left(\frac{\sqrt[4]{3}}{2}\right)-\frac{\pi}{6}]$ is increasing, $\cos\theta>\sqrt{\frac{3-\sqrt{3}}{4}}$ is decreasing and thus $s(\theta)$ is increasing.
\end{proof}

\begin{prop}
	\begin{equation*}
	t(\theta)=\frac{2\sqrt{3}\sin\left(\theta+\frac{\pi}{6}\right)+\sqrt{3}\cos\left(\theta+\frac{\pi}{3}\right)}{4\cos\theta\sin\left(\theta+\frac{\pi}{6}\right)+\sqrt{3}}\geq t(\frac{\pi}{6})
	\end{equation*}
	for $\arcsin\left(\frac{\sqrt[4]{3}}{2}\right)-\frac{\pi}{6}\leq\theta\leq\frac{\pi}{6}$.
	\label{prop:A3}
\end{prop}


\begin{proof}
Notice that
\begin{equation}
t(\theta)=\frac{3\sin(\theta+\frac{\pi}{3})}{4\sin^2(\theta+\frac{\pi}{3})+\sqrt{3}-1}.
\end{equation}
While $\theta\in[\arcsin\left(\frac{\sqrt[4]{3}}{2}\right)-\frac{\pi}{6},\frac{\pi}{6}]$ is increasing, $\sin(\theta+\frac{\pi}{3})>\frac{\sqrt{3}-1}{4}$ is increasing and thus $t(\theta)$ is decreasing.
\end{proof}

\begin{prop}
	\begin{equation*}
	t(\theta)\geq t(0)
	\end{equation*}
	for $0\leq \theta\leq \frac{\pi}{20}$ where
	\begin{equation*}
	\frac{t(\theta)-\frac{\sin\frac{3\pi}{10}\sin\frac{2\pi}{5}}{\sin(\frac{3\pi}{10}+\theta)}}{t(\theta)-\cos(\theta-\frac{2\pi}{5})}=\frac{\frac{\sin\frac{\pi}{5}\sin\frac{3\pi}{10}}{\sin(\frac{3\pi}{10}-\theta)}-\frac{\sin\frac{3\pi}{10}\sin\frac{2\pi}{5}}{\sin(\frac{3\pi}{10}+\theta)}}{\cos(\theta+\frac{\pi}{5})-\cos(\theta-\frac{2\pi}{5})}.
	\end{equation*}
	\label{prop:A4}
\end{prop}


\begin{proof}
Notice that
\begin{equation}
t(\theta)=\frac{2\sin\frac{3\pi}{10}\cos\frac{\pi}{10}\cos\theta}{2\cos^2\theta+\sin\frac{3\pi}{5}-\cos\frac{3\pi}{5}-1}.
\end{equation}
While $\theta\in[0,\frac{\pi}{20}]$, $\cos\theta>\sqrt{\frac{1-\sin\frac{3\pi}{5}+\cos\frac{3\pi}{5}}{2}}$ is decreasing and thus $t(\theta)$ is increasing.
\end{proof}


\begin{prop}
	\begin{equation}
	\cos(\theta-\frac{\pi}{5})>\frac{\sin\frac{\pi}{10}\sin\frac{\pi}{5}}{\sin(\frac{\pi}{10}+\theta)}
	\end{equation}
	for $0\leq\theta\leq\frac{\pi}{20}$.
	\label{prop:A5}
\end{prop}

\begin{proof}
It is equivalent to
\begin{equation}
\begin{split}
& \sin(\frac{\pi}{10}+\theta)\cos(\theta-\frac{\pi}{5})>\sin\frac{\pi}{10}\sin\frac{\pi}{5}\\
\Leftrightarrow & \sin(2\theta-\frac{\pi}{10})+\sin\frac{3\pi}{10}>2\sin\frac{\pi}{10}\sin\frac{\pi}{5}\\
\Leftrightarrow & \sin(2\theta-\frac{\pi}{10})>2\sin\frac{\pi}{10}\sin\frac{\pi}{5}-\sin\frac{3\pi}{10}.
\end{split}
\end{equation}
Since $\sin(2\theta-\frac{\pi}{10})$ is increasing for $0\leq\theta\leq\frac{\pi}{20}$, and
\begin{equation}
\sin\frac{3\pi}{10}-\sin\frac{\pi}{10}-2\sin\frac{\pi}{10}\sin\frac{\pi}{5}>0,
\end{equation}
the inequality holds for $0\leq\theta\leq\frac{\pi}{20}$.
\end{proof}

\begin{prop}
	\begin{equation}
	\cos(\theta+\frac{\pi}{5})>\frac{\sin\frac{\pi}{5}\sin\frac{3\pi}{10}}{\sin(\frac{3\pi}{10}-\theta)}
	\end{equation}
	for $0\leq\theta\leq\frac{\pi}{20}$.
	\label{prop:A6}
\end{prop}

\begin{proof}
It is equivalent to
\begin{equation}
\begin{split}
& \sin(\frac{3\pi}{10}-\theta)\cos(\theta+\frac{\pi}{5})>\sin\frac{\pi}{5}\sin\frac{3\pi}{10}\\
\Leftrightarrow & \sin\frac{\pi}{2}+\sin(\frac{\pi}{10}-2\theta)>2\sin\frac{\pi}{5}\sin\frac{3\pi}{10}\\
\Leftrightarrow & \sin(\frac{\pi}{10}-2\theta)>2\sin\frac{\pi}{5}\sin\frac{3\pi}{10}-\sin\frac{\pi}{2}.
\end{split}
\end{equation}
Since $\sin(\frac{\pi}{10}-2\theta)$ is decreasing for $0\leq\theta\leq\frac{\pi}{20}$, and
\begin{equation}
0>2\sin\frac{\pi}{5}\sin\frac{3\pi}{10}-\sin\frac{\pi}{2},
\end{equation}
the inequality holds for $0\leq\theta\leq\frac{\pi}{20}$.
\end{proof}

\begin{prop}
	\begin{equation}
	\cos(\theta-\frac{2\pi}{5})<\frac{\sin\frac{3\pi}{10}\sin\frac{2\pi}{5}}{\sin(\frac{3\pi}{10}+\theta)}
	\end{equation}
	for $0\leq\theta\leq\frac{\pi}{20}$.
	\label{prop:A7}
\end{prop}

\begin{proof}
It is equivalent to
\begin{equation}
\begin{split}
& \sin(\frac{3\pi}{10}+\theta)\cos(\theta-\frac{2\pi}{5})<\sin\frac{3\pi}{10}\sin\frac{2\pi}{5}\\
\Leftrightarrow & \sin(2\theta-\frac{\pi}{10})+\sin\frac{7\pi}{10}<2\sin\frac{3\pi}{10}\sin\frac{2\pi}{5}\\
\Leftrightarrow & \sin(2\theta-\frac{\pi}{10})<2\sin\frac{3\pi}{10}\sin\frac{2\pi}{5}-\sin\frac{7\pi}{10}
\end{split}
\end{equation}
Since $\sin(2\theta-\frac{\pi}{10})$ is increasing for $0\leq\theta\leq\frac{\pi}{20}$, and
\begin{equation}
0<2\sin\frac{3\pi}{10}\sin\frac{2\pi}{5}-\sin\frac{7\pi}{10},
\end{equation}
the inequality holds for $0\leq\theta\leq\frac{\pi}{20}$.
\end{proof}

\begin{prop}
	\begin{equation}
	\cos(\theta+\frac{2\pi}{5})<\frac{\sin\frac{2\pi}{5}}{\cos\theta}
	\end{equation}
	for $0\leq\theta\leq\frac{\pi}{20}$.
	\label{prop:A8}
\end{prop}

\begin{proof}
It is equivalent to
\begin{equation}
\begin{split}
& \cos\theta\cos(\theta+\frac{2\pi}{5})<\sin\frac{2\pi}{5}\\
\Leftrightarrow & \cos(2\theta+\frac{2\pi}{5})+\cos\frac{2\pi}{5}<2\sin\frac{2\pi}{5}\\
\Leftrightarrow & \cos(2\theta+\frac{2\pi}{5})<2\sin\frac{2\pi}{5}-\cos\frac{2\pi}{5}.
\end{split}
\end{equation}
Since $\cos(2\theta+\frac{2\pi}{5})$ is decreasing for $0\leq\theta\leq\frac{\pi}{20}$, and
\begin{equation}
\cos\frac{2\pi}{5}<2\sin\frac{2\pi}{5}-\cos\frac{2\pi}{5},
\end{equation}
the inequality holds for $0\leq\theta\leq\frac{\pi}{20}$.
\end{proof}

\end{document}